\documentclass[a4paper,11pt,leqno]{amsart}
\usepackage[active]{srcltx}
\usepackage[dvips]{graphicx} 

\usepackage[mathscr]{eucal}
\usepackage{mathrsfs}
\usepackage{amscd}
\usepackage{amsfonts}
\usepackage{amsmath,amsthm,amssymb}
\usepackage{latexsym}
\usepackage[dvips]{graphics}
\numberwithin{equation}{section}
\usepackage{amsthm}
\theoremstyle{plain}
 \newtheorem{theorem}{Theorem}[section]
 \newtheorem*{theorem*}{Theorem}

 \newtheorem*{lemma*}{Lemma}
 \newtheorem{proposition}[theorem]{Proposition}
 
 \newtheorem*{fact*}{Fact}

 \newtheorem{lemma}[theorem]{Lemma}
 \newtheorem{corollary}[theorem]{Corollary}
\theoremstyle{remark}
 
 \newtheorem{remark}[theorem]{Remark}
 \newtheorem*{remark*}{Remark}
 \newtheorem*{acknowledgements}{Acknowledgements}
 
\numberwithin{equation}{section}
 \newtheorem{question}[theorem]{Question}

\newcommand{\C}{{\Bbb C}}
\newcommand{\R}{{\Bbb R}}
\newcommand{\Z}{{\Bbb Z}}

\renewcommand{\phi}{\varphi}
\renewcommand{\epsilon}{\varepsilon}

\newcommand{\op}{\operatorname}

\newcommand{\mc}[1]{{\mathcal #1}}

\newcommand{\argh}[1]{\op{argh}{#1}}
\newcommand{\re}[1]{\op{Re}{#1}}
\newcommand{\im}[1]{\op{Im}{#1}}
\newcommand{\inner}[2]{\left\langle{#1},{#2}\right\rangle}
\newcommand{\lnorm}[1]{N^2\!\left( #1 \right)}

\title[]{
Examples of entire zero-mean curvature graphs
of mixed-type in Lorentz-Minkowski space
via Konderak's formulas}

\author[]{Takeki Komatsu}
\address[Komatsu]{Department of Mathematical and Computing Sciences \\
  Institute of Science Tokyo \\
  2-12-1-W8-34, O-okayama Meguro-ku \\
  Tokyo 152-8552, Japan} 
\email{}

\author[]{Masaaki Umehara}
\address[Umehara]{%
Department of Mathematical and Computing Sciences,
Institute of science Tokyo,
2-12-1-W8-34, O-okayama Meguro-ku,
Tokyo 152-8552 Japan
}
\email{umehara@comp.isct.ac.jp}


\thanks{%
The second author
  is supported by Grant-in-Aid 
for Scientific Research
(B) No.23K20794.
}%

\begin{document}
\maketitle
\begin{abstract}
Using Konderak's representation formula,
we construct
an entire zero-mean curvature 
graph of mixed-type in Lorentz-Minkowski $3$-space $\R^3_1$
over a space-like plane, 
which does not belong to the class of \lq\lq Kobayashi surfaces".
We also point out the existence of an entire
zero-mean curvature 
graph of mixed-type in Lorentz-Minkowski space
over a light-like plane.
These examples suggest that 
entire mixed-type zero-mean curvature graphs 
contain an unexpectedly large number of interesting examples.
\end{abstract}

\section{Introduction} \label{sec:I} 

We denote by $\R^3_1$ Lorentz-Minkowski 3-space
and denote by $\inner{}{}$ the canonical Lorentzian inner product
 of the signature $(-++)$.
In this paper, we consider entire graphs having zero-mean curvature
(i.e.~ZMC) surfaces in $\R^3_1$,
which are called {\it entire zero-mean curvature graphs} 
(we call {\it entire ZMC-graphs} for short).
More precisely, in $\R^3_1$,
functions generating ZMC-graphs are defined on
a space-like plane, a time-like plane or a light-like plane (cf. Magid \cite{M}).
Therefore, there are three types of entire graphs, 
which are distinguished by the terms type~S, type~T and type~L, respectively.

Calabi \cite{C} showed that 
space-like entire ZMC-graphs of type S
are space-like planes. 
Moreover, the same conclusion holds
for entire ZMC-graphs of type S
with light-like points (cf. \cite{AHUY} and \cite{AUY}).
Furthermore, as shown in Cheng-Yau \cite{CY},
proper and complete space-like ZMC-immersions into $\R^3_1$
are space-like planes. In particular, entire 
space-like ZMC-graphs of type T and L are
also space-like planes as well as the case of type S.

On the other hand,
there are many time-like entire ZMC-graphs
of types S, T and L (cf. \cite{M}).
As shown in \cite{FKKRUY},
there are mixed-type entire ZMC-graphs of type S (i.e.
containing space-like points and time-like points at the
same time), which are called {\it Kobayashi surfaces}
(the precise definition is in \cite[Definition 3.2]{FKKRUY}).
In Akamine-Umehara-Yamada \cite{AUY}, 
the three questions on embedded ZMC-surfaces are posed.
Especially,  the third problem is 
a question of whether there 
exists an entire ZMC-graph of type S, 
which is a mixed-type other than Kobayashi surfaces. 
In this paper, we give an answer to this question
by constructing an example
(the first and second problems in \cite{AUY} are still unsolved):

\begin{theorem}\label{thm:173}
There exists a mixed-type entire ZMC-graph of 
type S which is not a Kobayashi surface.
\end{theorem}
 
To construct space-like maximal surfaces in $\R^3_1$, 
there are two Weierstrass type formulas
given by Kobayashi \cite{Kob}. 
Each of them produces two different maximal surfaces, 
taking the real part or the imaginary
part of a holomorphic map induced by 
a pair of holomorphic functions called \lq\lq Weierstrass data". 
Thus, for each Weierstrass data,
we can construct four different maximal surfaces, in general.
Since some of them may have non-trivial analytic extensions
containing time-like points,
this can be one way to prove the theorem 
(cf. Remark~\ref{rmk:space-like}). 

Similar to the space-like case, 
there are formulas (given by Konderak \cite{Kon})
producing time-like zero mean curvature surfaces,
each of which is obtained by
taking the real part or the imaginary
part of a para-holomorphic map induced by 
a pair of para-holomorphic functions.
The (second) purpose of this paper is to 
construct several examples 
of ZMC-surfaces of the Scherk type 
using Konderak's formulas and prove 
Theorem \ref{thm:173} as the consequence.
It is well-known that
poles of a meromorphic function are isolated, 
but in the para-holomorphic category,
the points corresponding to poles are
not isolated (this phenomenon is due to the existence of non-invertible
nonzero paracomplex numbers).
By this reason, time-like zero-mean curvature surfaces
produced by Konderak's formulas might not be  globally defined,
in general. 

In this paper, we will overcome this difficulty as follows:
Each surface obtained by Konderak's formulas 
consists of several connected components in general. 
Fix one of the connected components $C$ arbitrarily.
We will try to find an implicit function form
of $C$ with the expectation that 
it gives an analytic  extension of $C$ in $\R^3_1$.
Especially when the number of the connected components is more than one,
several globally defined zero mean curvature surfaces associated with the 
connected components might be obtained.

To demonstrate this, we use Weierstrass data corresponding to the
Scherk surface which are given as pairs of para-holomorphic functions.
we apply them to Konderak's formulas. Then we
obtain eight ZMC-surfaces with four different congruent classes
(cf. Remark \ref{rmk;six}).
One of them denoted by $\mc S'_1$
coincides with the Scherk-type surface $\mc K_4$ obtained 
by Kobayashi's formula and 
used in the proof of Theorem \ref{thm:173}.
In connection with the above theorem, we also show 
the following:

\begin{proposition}\label{thm:183}
There exists a mixed-type entire ZMC-graph of type L,
which is a Kobayashi surface.
\end{proposition}

Based on Theorem \ref{thm:173} and Proposition \ref{thm:183} in this paper 
(in addition to the two remaining problems in \cite{AUY}),
 we would like to pose the following question:

\begin{question}
{\it Regarding the example $($cf. \eqref{eq:S4}$)$ 
in this paper, is there a systematic way to
construct mixed-type entire ZMC-graphs of type S other than Kobayashi surfaces?}
\end{question}

To the best of the authors' knowledge, all examples 
of entire ZMC-graphs induced by Kobayashi surfaces 
have connected space-like parts with finite total curvature,
and so, they are not periodic.
So, this question is closely related to following:

\begin{question}
{\it Is there a systematic way to 
construct periodic mixed-type entire ZMC-graphs of type S? }
\end{question}

Partial answers to this question  were recently found by 
Kim and Ogata \cite[Figure 2, Eq (3.14) and Eq (3.17)]{O}.
More precisely, they construct two families of singly periodic 
mixed-type entire ZMC-graphs of type S other than the example 
in this paper. On the other hand, relating to
Proposition~\ref{thm:183},
the following two questions are raised:

\begin{question}
  {\it Except for the example given in this paper, 
  are there any other mixed-type entire ZMC-graphs of type L?}
\end{question}

\begin{question}
  {\it Is there a mixed-type entire ZMC-graphs of type~T?}
\end{question}

The authors are unaware of any such examples of type~T.

\section{Konderak's representation formulas}

\subsection*{Properties of the logarithmic function on $\check \C$}
A para-complex number 
(which is also known as a split-complex number or hyperbolic number)
 is written as $z=u+jv$
($u,v\in \R$), where $j$ is called the {\it hyperbolic unit}
satisfying $j^2=1$.
The conjugate of $z$ is $\bar{z}:=u-jv$ 
and $\lnorm{z}:=z\bar z(=u^2-v^2)$.
We set
$$
\re(z):=\frac{z+\bar z}2(=u),\qquad
\im(z):=\frac{z-\bar z}{2j}(=v).
$$
We denote by $\check{\C}$ the set of para-complex numbers, which is
identified with $\R^2$.
If $z\in \check{\C}$ satisfies 
$\lnorm{z}>0$ (resp. $\lnorm{z}<0$),
it can be uniquely written as
\begin{align}\label{eq:245}
  z=\pm e^s(\cosh t+j\sinh t) \quad
\text{(resp. $z=\pm e^s(\sinh t+j\cosh t)$)}
\quad (s,t\in \R).
\end{align}
In each of these two cases, we denote $\argh z:=t$, which is
called the {\it hyperbolic argument} of $z$.
Then it holds that (cf. \eqref{eq:Log0} in the appendix)
\begin{equation}\label{eq:302}
\log z=s+jt=\log |z|+j \argh{z} \qquad \left(|z|
:=\sqrt{|\lnorm{z}|}\right).
\end{equation}
To avoid confusion with the usual absolute value symbol, 
the notation $\log |z|$ will not be used in this paper.
If $\alpha,\beta\in \check \C$ satisfies 
$N^2(\alpha)\ne 0$ and
$N^2(\beta)\ne 0$, 
we have
\begin{align}\label{eq-1}
& \argh{\left(\frac{\alpha}{\beta}\right)}
=\argh{\left (\alpha\bar{\beta} \right)}.
\end{align}

\subsection*{The four Weierstrass-type representation formulas}

For smooth real-valued functions 
 $X(u,v)$ and $Y(u,v)$ 
on a domain $U$ of $\check{\C}$,
a $\check \C$-valued function 
$f(u,v)=X(u,v)+jY(u,v)$
 is defined.
If it satisfies the para-Cauchy-Riemann equations
\begin{equation}\label{eq:PCR}
X_u=Y_v,\qquad Y_u=X_v 
\end{equation}
on $U$,
we call $f$ a {\it para-holomorphic function} on $U$.
Basic properties of para-holomorphic functions are given in the appendix.
Let $g$ and $\omega$ be two para-holomorphic functions on $U$ such that
  \begin{align}\label{eq:ds2}
    ds^2=-(1-\lnorm{g})^2\lnorm{\omega} (du^2-dv^2)
  \end{align}
gives a Lorentzian metric on $U$. 
We 
assume that $U$ is simply connected  and
fix a point $z_0\in U$.
Then (see Proposition \ref{eq:909} in the appendix 
for the definition of integrals)
   \begin{align}\label{eq:W1}
    F_1(z)&:=\re\int_{\!\!z_0}^{z}(-1-g^2,j(1-g^2),2g)\omega dz, \\
\label{eq:W2}
    F_2(z)&:=\im\int_{\!\!z_0}^{z}(-1-g^2,j(1-g^2),2g)\omega dz
  \end{align}
are time-like ZMC surfaces defined on $U$
whose first fundamental forms are $ds^2$ and $-ds^2$, respectively.
The second surface $F_2$ is called the {\it conjugate} surface of $F_1$.
We call the formula \eqref{eq:W1} (resp. \eqref{eq:W2}) 
the {\it first $($resp. second$)$
representation formula}.
The first formula 
\eqref{eq:W1} 
is essentially the same as Konderak's.
In fact, after changing the coordinates $(t,x,y)$ to $(y,x,-t)/2$
and replacing $(g,\omega)$ by $(jg,j\omega)$,
the first formula coincides with \cite[(3.14)]{Kon}.
On the other hand, if we replace $(g,\omega)$ in \eqref{eq:W1}-\eqref{eq:W2} by
$\left((g-1)/(g+1),(1+g)^2\omega/2\right)$,
then we obtain 
  \begin{align}\label{eq:W3}
    F_3(z)&:=\re\int_{\!\!z_0}^{z}(-1-g^2,2jg,-1+g^2)\omega dz,\\
\label{eq:W4} 
    F_4(z)&:=\im\int_{\!\!z_0}^{z}(-1-g^2,2jg,-1+g^2)\omega dz,
  \end{align}
where  
the first fundamental form of $F_3$ is given by
  \begin{align*}
    ds^2=-(g+\bar{g})^2\lnorm{\omega}(du^2-dv^2),
  \end{align*}
and the first fundamental form of $F_4$ is $-ds^2$.
The surface $F_3$ (resp.~$F_4$)
is called the {\it dual} (resp.~{\it conjugate dual}) surface 
of $F_1$. 
Moreover, we call the formula producing $F_3$ (resp.~$F_4$) 
the {\it third representation formula} 
(resp.~the {\it fourth representation formula}).
The third formula \eqref{eq:W3}
is essentially the same as Konderak's.
In fact, after changing the coordinates $(t,x,y)$ to $(x,-y,-t)/2$
and replacing $(g,\omega)$ by $(jg,\omega)$,
the third formula coincides with \cite[(3.11)]{Kon}.

\begin{figure}[h!]
  \begin{center}
    \includegraphics[height=36mm]{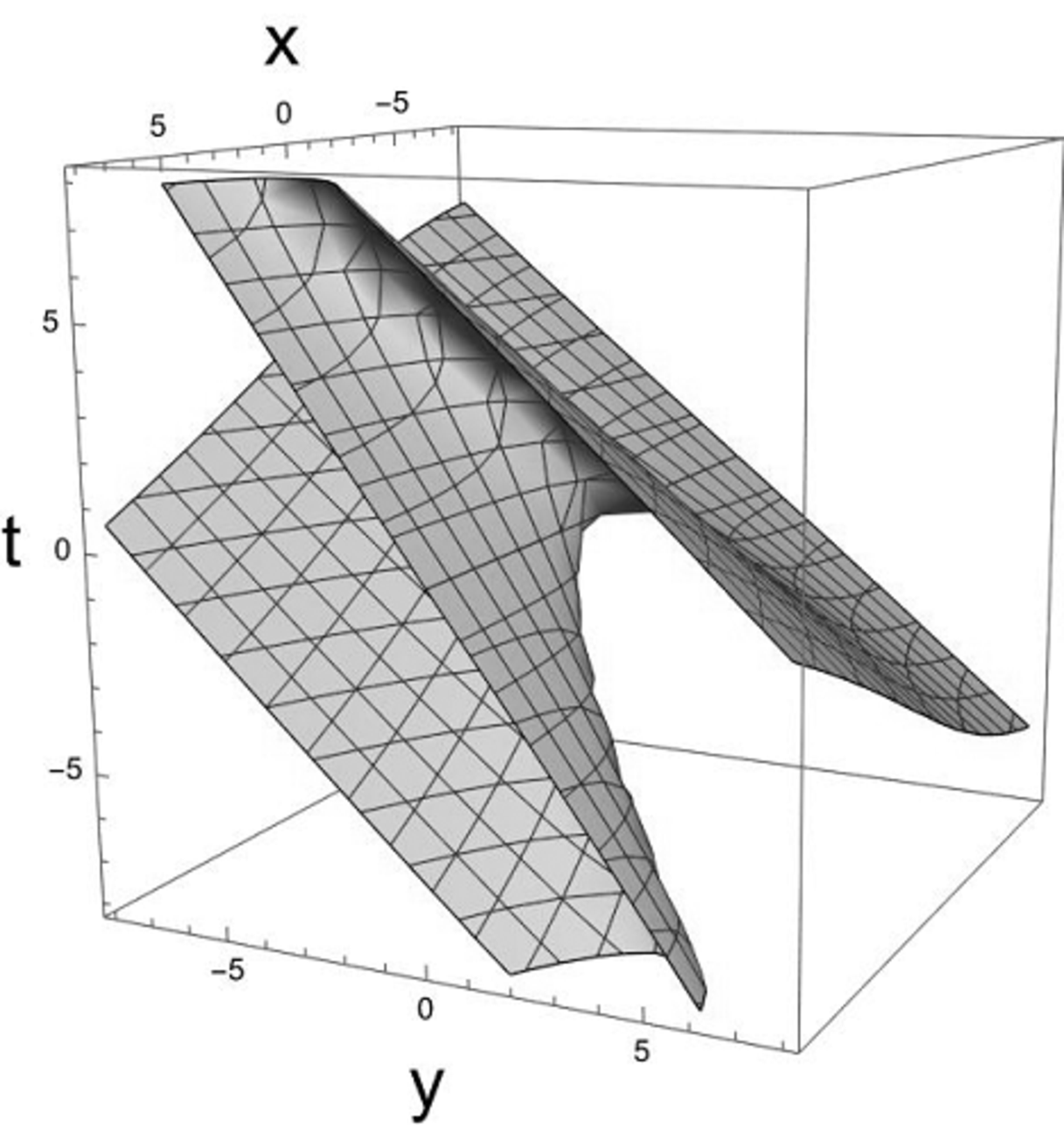}\quad
    \includegraphics[height=36mm]{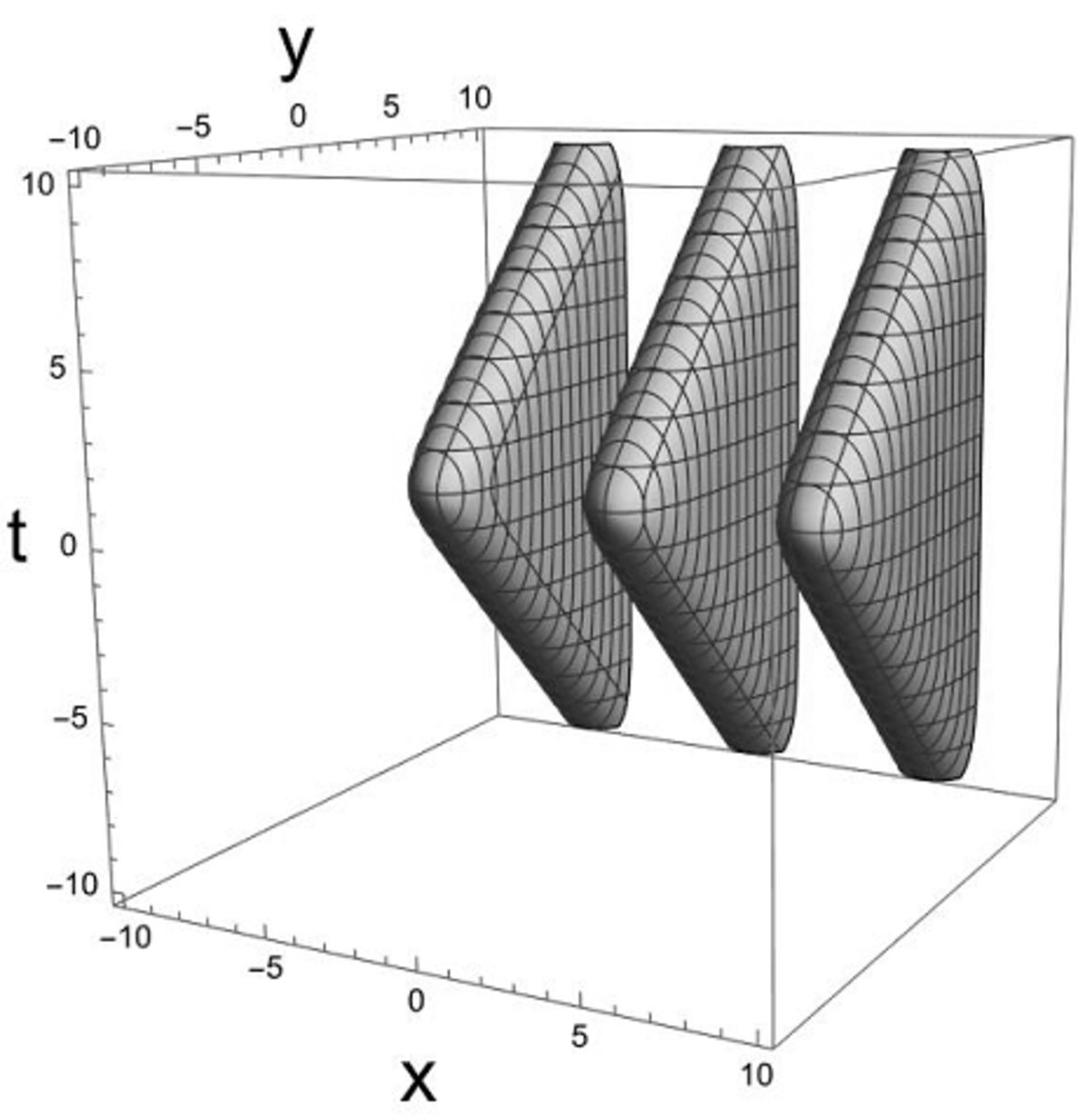}\quad
    \includegraphics[height=36mm]{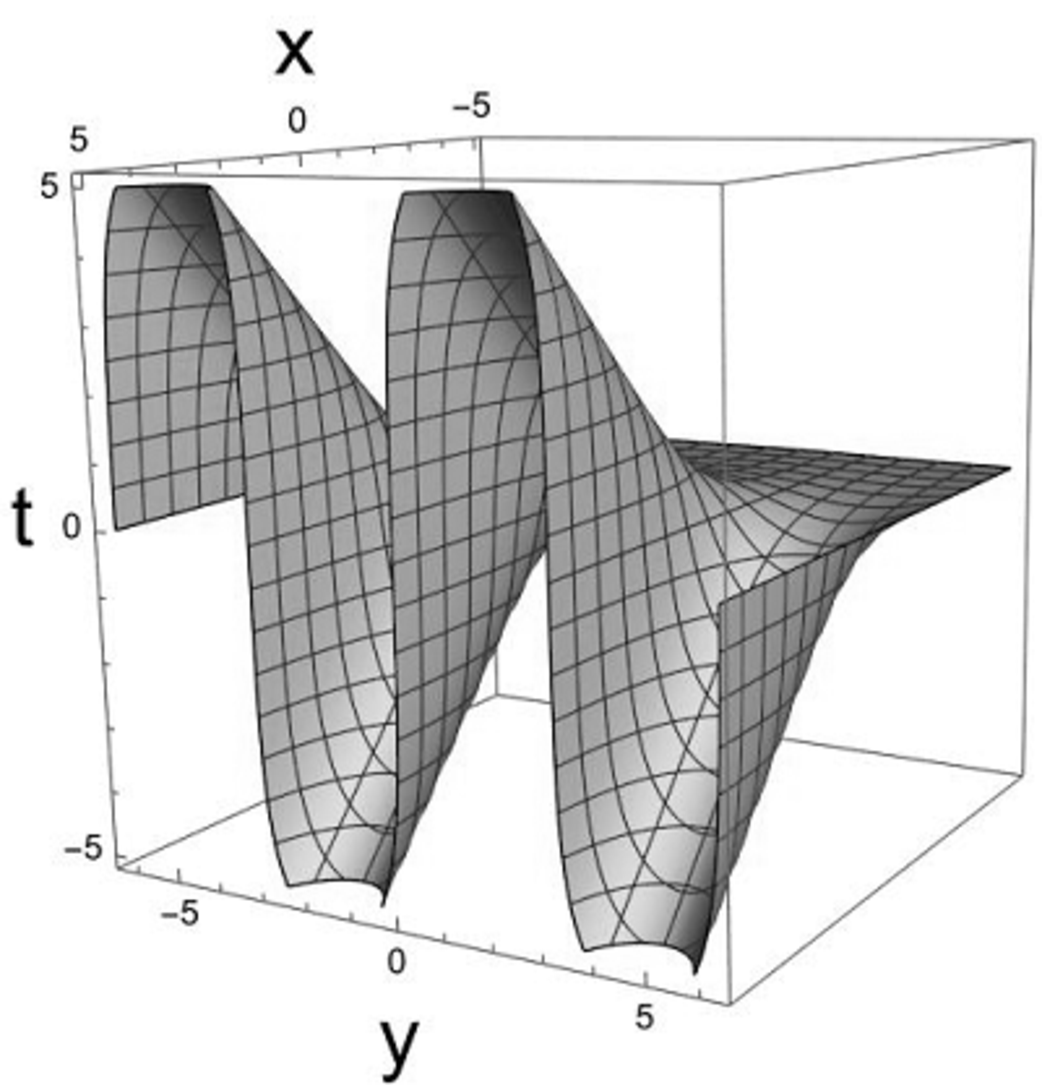}
    \caption{The Enneper-type surface $\mc E_4$ (left)
and the Scherk-type surfaces $\mc S_1$ (center) and $\mc S'_1$ (right)
}
    \label{fig1}
  \end{center}
\end{figure}

\subsection*{An interpretation of the four formulas}

We can set (cf. Proposition \ref{label:prop1092})
$$
g(z)=
\frac{g_1(u+v)+g_2(u-v)}2+j\frac{g_1(u+v)-g_2(u-v)}2
$$
and
$$
\omega(z)=
\frac{w_1(u+v)+w_2(u-v)}2+j\frac{w_1(u+v)-w_2(u-v)}2.
$$
By setting $x=(u+v)/2,\,\, y=(u-v)/2$,
we can write (cf. \eqref{eq:1110} in the appendix)
$$
g(z)=\epsilon_1 \hat g_1(x)+\epsilon_2 \hat g_2(y),\qquad
\omega(z)=\epsilon_1 \hat \omega_1(x)+\epsilon_2 \hat\omega_2(y),
$$
where $\epsilon_i$ ($i=1,2$) are para-complex numbers defined in \eqref{eq:1080}
and
\begin{align*}
&\hat g_1(x)=g_1(2x)=g_1(u+v),\qquad \hat g_2(y)=g_1(2y)=g_1(u-v), \\
&\hat w_1(x)=w_1(2x)=w_1(u+v),\qquad \hat w_2(y)=w_1(2y)=w_1(u-v).
\end{align*}
We assume that $g_1,g_2,\omega_1,\omega_2$ 
are defined on an interval of $\R$ containing the origin.
We can write $dz=2(\epsilon_1 dx+\epsilon_2 dy)$.
Using \eqref{eq1E} and the relations
$$
\op{Re}(\epsilon_1)=\op{Re}(\epsilon_2)=\frac12,\qquad
\op{Im}(\epsilon_1)=-\op{Im}(\epsilon_2)=\frac12,
$$
we can rewrite the formulas of $F_i$ ($i=1,2,3,4$).
For example, we have
\begin{align}\label{eq:413}
F_1&=\sum_{i=1}^2 \int_{0}^{x_i} \Big( 
-1-g_i(x_i)^2,
(-1)^{1+i}(1-g_i(x_i)^2),
2g_i(x_i) \Big) \omega_i(x_i)dx_i, \\
F_3&=\sum_{i=1}^2 \int_{0}^{x_i} \Big( 
-1-g_i(x_i)^2,
(-1)^{1+i}2g_i(x_i),-1+g_i(x_i)^2
\Big) \omega_i(x_i)dx_i,
\end{align}
where $dx_1:=dx$ and $dx_2:=dy$. 
The first formula
\eqref{eq:413} is essentially the same as \cite[Fact~A.7]{A}.
If we substitute $-\omega_2$ to $\omega_2$
in the above two formulas, we obtain $F_2$ and $F_4$, respectively.
In spite of these new expressions of $F_i$ ($i=1,2,3,4$),
the original the four formulas 
\eqref{eq:W1}, \eqref{eq:W2}, \eqref{eq:W3}, \eqref{eq:W4} 
are crucial in this paper,
since we need to apply
para-holomorphic calculus as in the appendix, individually.

\subsection*{An entire ZMC-graph of type L}
We first prove Proposition~\ref{thm:183}:

\medskip
\noindent
{\it Proof of Proposition~\ref{thm:183}.}
We set
\begin{equation}\label{eq:370}
  g:=z,\quad \omega:=1,
\end{equation}
which are the Weierstrass data of the Enneper surface in Euclidean 3-space.
Here, we think of \eqref{eq:370}
 as a pair of para-holomorphic functions, and
substitute it into the formulas \eqref{eq:W3} and \eqref{eq:W4},
we obtain
\begin{align*}
  F_3+jF_4=\left(-z-\frac{z^3}{3},jz^2,-z+\frac{z^3}{3}\right).
\end{align*}
Letting $z=u+jv$ and taking the imaginary part, 
we have
\begin{align}
F_3(u,v)&=\left(-\frac{u^3}{3}-u v^2-u,2 u v,\frac{u^3}{3}+u v^2-u\right),\\
F_4(u,v)&=\left(-u^2 v-\frac{v^3}{3}-v,u^2+v^2,u^2 v+\frac{v^3}{3}-v\right),
\end{align}
whose singular sets are $u^2-v^2=1$.
Moreover, we can write
\begin{equation*}
  \begin{aligned}
    F_4
           &=\left(-v-\frac{v^3}{3},v^2,-v+\frac{v^3}{3}\right)+s\left(-v,1,v\right),
  \end{aligned}
\end{equation*}
where $s:=u^2$.
This expression implies 
that the surface can be real analytically extended as  a ruled surface
(cf. \cite{Kob}). 
The whole image of this 
analytic extension can be expressed as 
(see Figure \ref{fig1}, left)
\begin{align}\label{eq;E4}
  \mc E_4:=
\left\{(t,x,y)\in\mathbb{R}^3_1\,\,;\,\,t-y=-\frac{(t+y)^3}{6}+x(t+y)\right\},
\end{align}
which is given in Kim, Lee and Yang \cite{KLY}.  
Moreover,
it has been shown in \cite[Example 3.8]{FKKRUY} 
that it is an example of a Kobayashi surface.
If we set
  $
  \eta:=t+y,\,\, \zeta:=t-y,
  $
then $(x,\eta,\zeta)$ gives a new coordinate system of $\R^3_1$ 
  such that the $x\eta$-plane is light-like.
  Moreover, $\mc E_4$ coincides with the graph of the function
  $\zeta:=x\eta-\eta^3/6$.
\qed

\begin{remark}
The image of $F_3$ is
a subset of
$$
\mc E_3:=\left\{(t,x,y)\in\mathbb{R}^3_1\,\,;\,\,t^2-y^2=
\frac{(t+y)^4}{12}+x^2\right\}.
$$
This expression is given in \cite{FKKRSUYY}.
The subset $\mc E_3$ contains 
a light-like line $L:=\{(t,0,-t)\,;\, t\in \R\}$ but not contains any
space-like points.
Similarly, if we substitute \eqref{eq:370} into the first and second formulas,
we obtain
\begin{align*}
F_1(u,v)&=\left(-u-\frac{u^3}{3}-uv^2,v-\frac{v^3}3-u^2v,u^2+v^2\right), \\
F_2(u,v)&=\left(-v-\frac{v^3}{3}-u^2v,u-\frac{u^3}{3}-uv^2,2uv\right),
\end{align*}
which are called the {\it time-like Enneper surface} and its dual.
These two surfaces have singular points along the set $u^2-v^2=1$.
From these examples, we can see that
the four representation formulas induce four different surfaces, in general.
\end{remark}

\section{ZMC-surfaces associated with the Scherk surface}

To obtain implicit function forms of Scherk-type surfaces
using Konderak's representation formulas, 
the following identities are useful:
\allowdisplaybreaks{
\begin{align}\label{Ar1}
&\cosh \log\left(\frac{a}b\right)=\frac{a^2+b^2}{2ab},\quad
\sinh \log\left(\frac{a}b\right)=\frac{a^2-b^2}{2ab} 
\qquad (a,b>0), \\
& \cos(\arctan a)
=\frac{1}{\sqrt{1+a^2}},\quad \sin(\arctan a)=\frac{a}{\sqrt{1+a^2}}
\qquad (a\in \R), \\ \label{eq:ATab}
& \cos(\arctan a\pm \arctan b)=\frac{1\mp ab}{\sqrt{1+a^2}\sqrt{1+b^2}} \qquad (a,b\in \R), \\
& \sin(\arctan a\pm \arctan b)=\frac{a\pm b}{\sqrt{1+a^2}\sqrt{1+b^2}} \qquad (a,b\in \R), \\
&\cosh (\op{argh}z)=\frac{|\re(z)|}{\sqrt{|\lnorm{z}|}},\quad
|\sinh (\op{argh} z)|=\frac{|\im(z)|}{\sqrt{|\lnorm{z}|}}\qquad
(\text{if $\lnorm{z}>0$}), \\
\label{Ar2}
&\cosh (\op{argh}z)=\frac{|\im(z)|}{\sqrt{|\lnorm{z}|}},\quad
|\sinh (\op{argh} z)|=\frac{|\re(z)|}{\sqrt{|\lnorm{z}|}}\qquad
(\text{if $\lnorm{z}<0$}).
\end{align}}

We set
\begin{align}\label{eq:303}
  g:=-z,\qquad \omega:=\frac{1}{z^4-1},
\end{align}
which are well-known as the Weierstrass data of Scherk's surface 
in Euclidean 3-space. 
We think \eqref{eq:303}
 is a pair of para-holomorphic functions.
Substituting \eqref{eq:303} to the first and second representation formulas
(cf. \eqref{eq:W1} and \eqref{eq:W2}).
Then
\begin{align*}
  \Big(-1-g^2,j(1-g^2),2g\Big)\omega
=\frac{1}{2}
\left(\frac{1}{z+1}-\frac{1}{z-1},\frac{\-2j}{z^2+1},
\frac{2z}{z^2+1}-\frac{2z}{z^2-1}\right).
\end{align*}
By 
\eqref{eq:W1},
\eqref{eq:W2}, \eqref{eq:Log0},
\eqref{eq:log}  and \eqref{eq:TanI},
we have
\begin{equation}\label{eq:F1F2}
F_1+jF_2=
\frac{1}{2}\Big(\log \! A(z),
-2j \arctan z,\log \! A(z^2)\Big)
\quad
\left(A(z):=\frac{z+1}{z-1}\right).
\end{equation}
By setting $z=u+jv$ ($u,v\in \R$),
the following two identities are obtained
\begin{align}
\label{N2}
&\lnorm{z+1}+\lnorm{z-1}=2(1-u^2-v^2)=2(1+\lnorm{z}).
\end{align}
According to 
\eqref{eq;447}
in the appendix,
we define
\begin{align}\nonumber
  &\arctan(u+jv)
 :=\frac{1}{2}\Big(
\arctan(u+v)+\arctan(u-v) \\ \label{eq:AT}
&\phantom{aaaaaaaaaaaaaaaaaaaa}+j\left(\arctan(u+v)-\arctan(u-v)\right)\Big).
\end{align}
Using this, we compute the real part of \eqref{eq:F1F2}:
If we write $2F_1=(T_1,X_1,Y_1)$,
then, by  \eqref{eq:302} and \eqref{eq:AT}, we have
\begin{align*}
&T_1=\log\sqrt{\left|\lnorm{A(z)}\right|},\quad
X_1=\arctan(u-v)-\arctan(u+v), \\
&Y_1=\log\sqrt{\left|\lnorm{A(z^2)}\right|}.
\end{align*}
By \eqref{eq:ATab} and
 \eqref{eq:N3}, we have 
\begin{equation}\label{eq:551}
\cos X_1
=\frac{1+\lnorm{z}}
{\sqrt{(1+(u-v)^2)(1+(u+v)^2)}} =\frac{1+\lnorm{z}}{\sqrt{\lnorm{z^2+1}}}.
\end{equation}
We set 
\begin{align}\label{Uepsilon0}
  D_+:=\{z\in \check \C\,;\, \lnorm{A(z)}>0\},\quad
  D_-:=\{z\in \check \C\,;\, \lnorm{A(z)}<0\},
\end{align}
which are non-empty subsets of $\check \C$.
We consider the case
$z\in D_+$:
Then the signs of $\lnorm{z+1}$ and $\lnorm{z-1}$ are the same.
So, we have 
\begin{align}\nonumber
\cosh T_1&=
\cosh \log \left(\sqrt{\frac{|\lnorm{z+1}|}{|\lnorm{z-1}|}}\right) 
=\frac{|\lnorm{z+1}|+|\lnorm{z-1}|}{2\sqrt{|\lnorm{z+1}\lnorm{z-1}|}} \\ \nonumber
&=\frac{|\lnorm{z+1}+\lnorm{z-1}|}{2\sqrt{|\lnorm{z+1}\lnorm{z-1}|}}=
\frac{|1+\lnorm{z}|}{\sqrt{|\lnorm{z+1}\lnorm{z-1}|}} \\
&=\epsilon \cos X_1 e^{Y_1}, \label{eq:565}
\end{align}
where $\epsilon\in \{+,-\}$ is the sign of $1+N^2(z)$.
Thus, the set $2F_1(D_+)$ is contained in
$$
\mathbb S_1:=\{(t,x,y)\in \R^3_1\,;\, \cosh^2 t=e^{2y}\cos^2 x \},
$$
which has no singular points and consists only of time-like points.
The set $\mathbb S_1$ consists of
the two subsets
\begin{align*}
&\mc S_1:=\{(t,x,y)\in \R^3_1\,;\, \cosh t=e^{y}\cos x \}, \\
&\tilde {\mc S}_1:=\{(t,x,y)\in \R^3_1\,;\, \cosh t=-e^{y}\cos x \},
\end{align*}
that are congruent. Every connected component of $\mathbb S_1$ 
is congruent to the set
$$
\{(t,x,y)\in \R^3_1\,;\, \cosh t=e^{y}\cos x,\,\, |x|<\pi \} (\subset \mc S_1). 
$$
The figure of $\mc S_1$  is shown in the middle image of Figure~\ref{fig1}.

We next consider the case
$z\in D_-$:
Then the signs of $\lnorm{z+1}$ and $\lnorm{z-1}$ are different.
By setting $t_1:=T_1$, $y:=Y_1$ and $x:=X_1$ 
(resp. $x:=X_1+\pi$)
if $T_1>0$ (resp. $T_1<0$),
we have
\begin{align} \nonumber
|\sinh T_1|&=
\sinh \log \left(\sqrt{\frac{|\lnorm{z+1}|}{|\lnorm{z-1}|}}\right) 
=\left|\frac{|\lnorm{z+1}|-|\lnorm{z-1}|}{2\sqrt{|\lnorm{z+1}\lnorm{z-1}|}}\right| \\
\nonumber
&=\frac{|\lnorm{z+1}+\lnorm{z-1}|}{2\sqrt{|\lnorm{z+1}\lnorm{z-1}|}}=
=\frac{\pm\Big|1+\lnorm{z}\Big|}{\sqrt{|\lnorm{z+1}\lnorm{z-1}|}}\\
&=
\epsilon \cos X_1 e^{Y_1},\label{eq:586}
\end{align}
where $\epsilon\in \{+,-\}$ is the sign of $1+N^2(z)$.
Thus, the set $2F_1(D_-)$ is a subset of
\begin{equation}\label{eq:S4}
\mathbb S'_1:=\{(t,x,y)\in \R^3_1\,;\, \sinh^2 t=e^{2y}\cos^2 x \}.
\end{equation}
The set $\mathbb S'_1$ has two connected components
\begin{align*}
&\mc S'_1:=\{(t,x,y)\in \R^3_1\,;\, \sinh t=e^{y}\cos x \}, \\
&\tilde{\mc S}'_1:=\{(t,x,y)\in \R^3_1\,;\, \sinh t=-e^{y}\cos x \},
\end{align*}
which are congruent. The first set $\mc S'_1$ gives the
entire graph
\begin{equation}\label{eq:692}
t=\op{arcsinh}(e^{y}\cos x), 
\end{equation}
which is indicated in Figure~\ref{fig1}, right.
This surface is a singly periodic entire graph of mixed-type, which will be applied
to prove  \ref{thm:173}, later.
We next compute $F_2$.
If we write $2F_2=(T_2,X_2,Y_2)$,
then 
$$
T_2=\argh A(z),\quad
X_2=-\arctan(u+v)-\arctan(u-v),\quad
Y_2=\argh A(z^2).
$$
By 
\eqref{eq:N3},
we have 
\begin{equation}\label{eq:652}
\sin X_2=\frac{-2u}{\sqrt{1+(u+v)^2}\sqrt{1+(u-v)^2}}=\frac{-(z+\bar z)}
{\sqrt{\lnorm{z^2+1}}}.
\end{equation}
Since (cf. \eqref{eq:N3})
$$
\lnorm{A(z)}\,\lnorm{A(z^2)}
=
\frac{\lnorm{z^2+1}}
{\left(\lnorm{z-1}\right)^2}>0,
$$
$\lnorm{A(z^2)}$ is positive (resp. negative)
if $z\in D_+$ (resp. $z\in D_-$).
In particular, if $z\in D_+$, we have
(cf. \eqref{eq-1} and \eqref{Ar1})
\begin{align}\label{eq:624a}
|\sinh T_2|&
=
|\sinh \argh A(z)|
=\left|
\frac{\im((z+1)(\bar z-1))}{\sqrt{|\lnorm{z^2-1}|}}
\right|
=\frac{|j(\bar z-z)|}{\sqrt{|\lnorm{z^2-1}|}},
 \\
|\sinh Y_2|&= \label{eq:624b}
|\sinh \argh A(z^2)|
=\frac{|j(\bar z^2-z^2)|}{\sqrt{|\lnorm{z^4-1}|}} 
=\frac{|j(\bar z-z)(\bar z+z)|}{\sqrt{|\lnorm{z^4-1}|}}.
\end{align}
So if we set
$$
\mathbb S_2:=\{(t,x,y)\in \R^3_1\,;\, \sinh^2 y=\sinh^2 t \sin^2 x\},
$$
then $2F_2(D_+)$ is a subset of $\mathbb S_2$.
By definition, $\mathbb S_2$ consists of the two subsets
\begin{align*}
&\mc S_2:=\{(t,x,y)\in \R^3_1\,;\, \sinh t=e^{y}\cos x \}, \\
&\tilde{\mc S}_2:=\{(t,x,y)\in \R^3_1\,;\, \sinh t=-e^{y}\cos x \},
\end{align*}
which are congruent, The first subset $\mc S_2$ 
contains the light-like lines 
$\{(t,n \pi,\pm t)\,;\, t\in \R,\,\, n\in \Z\}$,
but does not contain any space-like points 
(see Figure~\ref{fig2}, left).

On the other hand, when $z\in D_-$, we have
(cf. \eqref{Ar2})
\begin{align}\label{eq:647a}
\cosh T_2&=\cosh \argh A(z)
=\frac{|\im((z+1)(\bar z-1))|}{\sqrt{|\lnorm{z^2-1}|}}
=\frac{|j(\bar z-z)|}{\sqrt{|\lnorm{z^2-1}|}}, \\
\label{eq:647b}
\cosh Y_2&=
\cosh \argh A(z^2)
=\frac{|j(\bar z^2-z^2)|}{\sqrt{|\lnorm{z^4-1}|}} 
=\frac{|j(\bar z-z)(\bar z+z)|}{\sqrt{|\lnorm{z^4-1}|}}. 
\end{align}
Thus, $2F_2(D_-)$ is a subset of
$$
\mathbb S'_2:=\{(t,x,y)\in \R^3_1\,;\, \cosh^2 y=\cosh^2 t \sin^2 x\}.
$$
By definition, $\mc S'_2$ consists of the two subsets
\begin{align*}
&\mc S'_2:=\{(t,x,y)\in \R^3_1\,;\, \cosh y=\cosh t \sin x \}, \\
&\tilde{\mc S}'_2:=\{(t,x,y)\in \R^3_1\,;\, \cosh y=-\cosh t \sin x \},
\end{align*}
which are congruent.
The figure of $\mc S'_2$ is  given in Figure~\ref{fig2}, right.
Furthermore, every connected component of $S'_2$ is congruent to
$$
\{(t,x,y)\in \R^3_1\,;\, \cosh y=\cosh t \sin x,\,\, x\in (0,\pi) \},
$$
which does not contain any space-like points, but does
contain light-like lines
$$
\{(t,\frac{\pi}2,\pm t)\,;\, t\in \R,\,\, n\in \Z\}.
$$

\begin{remark}
We set $g:=z$ and $\omega:=-1/z^2$,
which are the Weierstrass data of the catenoid in Euclidean 3-space.
We think they are para-holomorphic functions.
Since the singular set of 
$\omega(=-1/{z^2})$ is just the point $z(=u+jv)$ satisfying
$u^2-v^2=0$ and thus the non-singular part of $\omega$
has four connected components.
If $\lnorm{z}>0$ $($resp. $\lnorm{z}<0)$, then we can write
$$
z=\pm e^s(\cosh t+j \sinh t)\qquad
\Big(\text{\rm resp. } z=\pm e^s (\sinh t+j \cosh t)\Big),
$$
where $s,t\in \R$.
Using these expressions, 
the first and second formulas for $\lnorm{z}>0$
induce
\begin{align}
\mc C_1=\{(t,x,y)\,;\, t^2-x^2=\sinh^2 y\},\quad
\mc C_2=\{(t,x,y)\,;\, t=x\tanh y\}.
\end{align}
On the other hand,
the first and second formulas for $\lnorm{z}<0$
induce
\begin{align}
\mc C'_1=\{(t,x,y)\,;\, x^2-t^2=\cosh^2 y\},\quad
\mc C'_2=\{(t,x,y)\,;\, t\tanh y=x\}.
\end{align}
Similarly, we can compute
the third and fourth formulas for $\lnorm{z}>0$
and $\lnorm{z}<0$, respectively. Then the four surfaces
$\mc C_3,\,\mc C'_3,\mc C_4,\,\mc C'_4$ are obtained, but
\begin{itemize}
\item $\mc C_3$ is congruent to $\mc C'_1$,
\item $\mc C_4$ is congruent to $\mc C'_2$,
\item $C'_3$ is congruent to $\mc C_1$, and 
\item $C'_4$ is congruent to $\mc C_2$.
\end{itemize}
As a consequence,
the four Konderak's formulas induce the 
four non-congruent ZMC-surfaces.
$\mc C_1,\,\mc C'_1$ and $\mc C'_2$ have no space-like points,
and $\mc C_2$ is a mixed-type entire ZMC-graph over a space-like plane.
These surfaces are well-known:
$\mc C_1$ is given in \cite[Figure~1(b)]{FKKRSUYY},
$\mc C_2$ is given in Kobayashi \cite[Example 2.8]{Kob},
$\mc C'_1$ is given in
Akamine-Lopez \cite{AL}, and 
$\mc C'_2$ is given in Kim-Koh-Shin-Yang \cite[Example 2.9]{KKSY}.
\end{remark}

\begin{figure}
  \begin{center}
    \includegraphics[height=38mm]{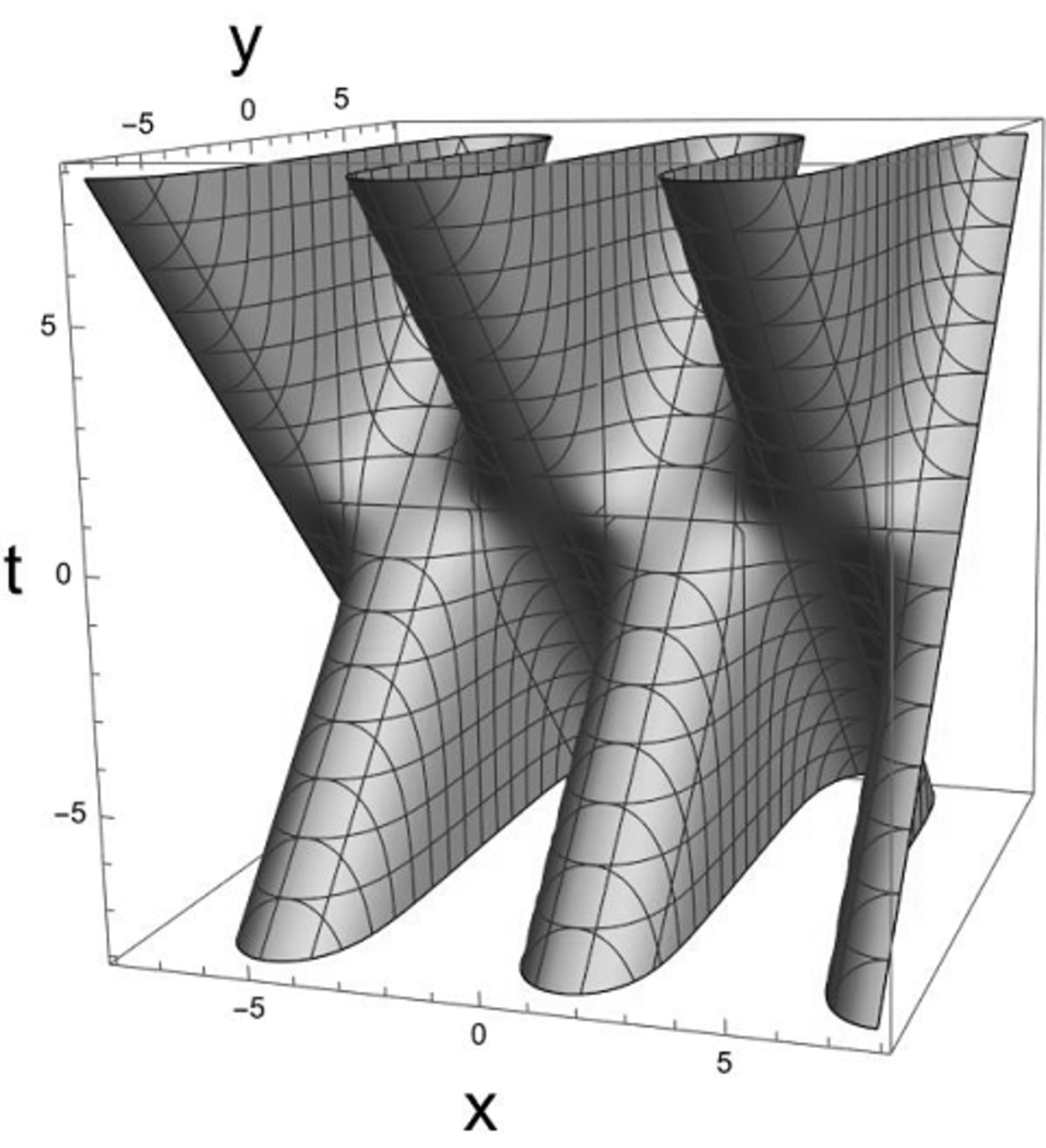} \qquad
    \includegraphics[height=38mm]{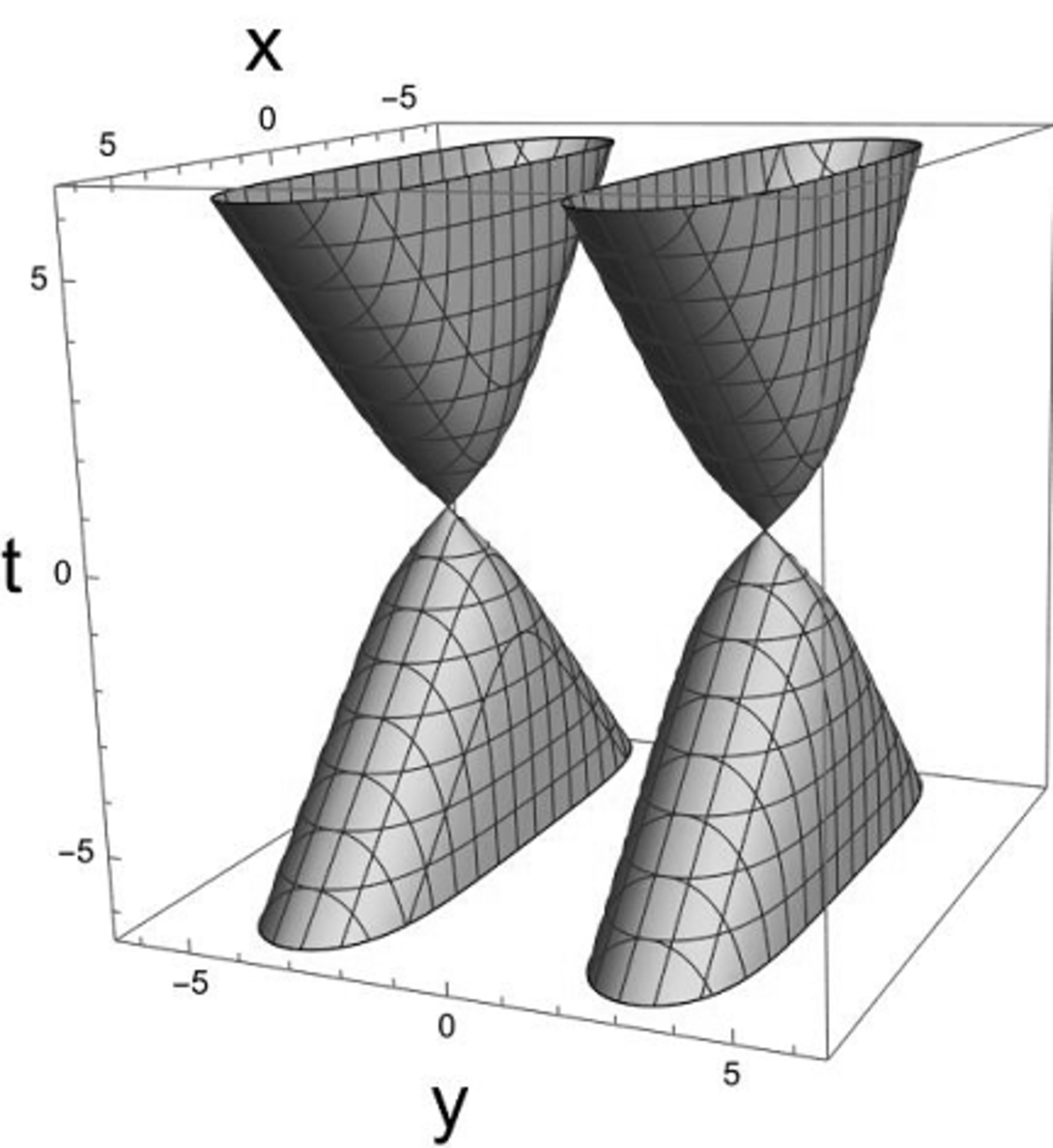}
    \caption{The Scherk-type surfaces $\mathcal S_2$ (left) and $\mathcal S'_2$ (left)}
    \label{fig2}
  \end{center}
\end{figure}

We next substitute 
\eqref{eq:303}
to Konderak's third and fourth formulas: 
Since
\begin{align*}
  (-1-g^2,2jg,-1+g^2)\omega=\frac{1}{2}\left(\frac{1}{z+1}-\frac{1}{z-1},
\frac{2jz}{z^2+1}-\frac{2jz}{z^2-1},\frac{2}{z^2+1}\right),
\end{align*}
letting $z=u+jv$ ($u,v\in \R$)
and integrating it,
we obtain
\begin{align}\label{eq:697}
  F_3+jF_4=\frac{1}{2}\left
(\log\left(\frac{z+1}{z-1}\right),j\log\left(\frac{z^2+1}{z^2-1}\right),2\arctan z\right).
\end{align}
Then we have (cf. \eqref{eq:AT})
\begin{align*}
F_3=\frac12 \left(
\log \sqrt{N^2(A(z))},\argh{A(z^2)},\arctan(u+v)+\arctan(u-v)
\right).
\end{align*}
If we write $2F_3=(T_3,X_3,Y_3)$,
then 
\begin{align*}
&T_3=\log\sqrt{\left|\lnorm{A(z)}\right|},\quad
X_3=\argh{A(z^2)}, \\
&Y_3=\arctan(u+v)+\arctan(u-v).
\end{align*}
Imitating the computations
in \eqref{eq:551}, we have
\begin{equation}\label{eq:750}
\cos Y_3=
\frac{1-\lnorm{z}}{\sqrt{|\lnorm{z^2+1}|}}.
\end{equation}
Imitating the computations
in \eqref{eq:565} and
\eqref{eq:586},
we have
\begin{equation}\label{eq:756}
\frac{|1+N^2(z)|}{\sqrt{|\lnorm{z^2-1}|}}=
\begin{cases}
\cosh T_3 & \text{if $z\in D_+$}, \\
|\sinh T_3| & \text{if $z\in D_-$}.
\end{cases}
\end{equation}
On the other hand, 
imitating the computations of
\eqref{eq:624a} and 
\eqref{eq:647a},
we have
$$
\frac{|1-N^2(z)^2|}{\sqrt{|\lnorm{z^4-1}|}}
=
\begin{cases}
\cosh X_3 & \text{if $z\in D_+$}, \\
|\sinh X_3| & \text{if $z\in D_-$}.
\end{cases}
$$
Thus,
$2F_3(D_+)$ is contained in the set
$$
\mathbb S_3=\{(t,x,y)\in \R^3_1\,;\, \cosh^2 t \cos^2 y=\cosh^2 x\},
$$
and $2F_3(D_-)$ is a subset of
$$
\mathbb S'_3=\{(t,x,y)\in \R^3_1\,;\, \sinh^2 t \cos^2 y=\sinh^2 x\}.
$$
We remark that $\mathbb S_3$ (resp. $\mathbb S'_3$) is
congruent to $\mathbb S'_2$ (resp. $\mathbb S_2$).

We next set $2F_4=(T_4,X_4,Y_4)$,
then 
\begin{align*}
&T_4=\argh{A(z)},\quad
X_4=\log\sqrt{\left|\lnorm{A(z)}\right|}, \\
&Y_4=\arctan(u+v)-\arctan(u-v).
\end{align*}
By \eqref{eq:624a} and \eqref{eq:647a},
we have
$$
\frac{|j(\bar z-z)|}{\sqrt{|\lnorm{z^2-1}|}}
=
\begin{cases}
|\sinh T_4| & \text{if $z\in D_+$}, \\
\cosh T_4 & \text{if $z\in D_-$}.
\end{cases}
$$
Applying
\eqref{eq:N3}
and imitating 
\eqref{eq:652},
we have
$$
|\sin Y_4|=\frac{|j(\bar z-z)|}{\sqrt{|N^2(z^2+1)|}}.
$$
Thus, $2F_4(D_+)$ is a subset of
\begin{equation}
\mathbb S_4:=\left\{(t,x,y)\in\mathbb{R}^3_1\,\,;\,\,\sinh^2 t=e^{2x}\sin^2 y\right\},
\end{equation}
and 
$2F_4(D_-)$ is a subset of
$$
\mathbb  S'_4:=\left\{(t,x,y)\in\mathbb{R}^3_1\,\,;\,\,\cosh^2 t=e^{2x}\sin^2 y\right\}.
$$
We remark that $\mathbb S_4$ (resp. $\mathbb S'_4$) is
congruent to $\mathbb S'_1$ (resp. $\mathbb S_1$).

\begin{remark}\label{rmk;six}
As seen in the above, by Konderak's four formulas with the
data \eqref{eq:303}, we constructed 
eight ZMC-surfaces $\mathcal S_k$ and $\mathcal S'_{k}$ $(k=1,\ldots,4)$
in total,
which are variants of the Scherk surface. 
Amongst them,
there are four mutually non-congruent 
connected ZMC-surfaces
\begin{align*}
\hat S_1&:=\{\cosh t=e^y \cos x,\,\, |x|<\pi/2 \},\\
\mc S'_1&:=\{\sinh t=e^{y}\cos x \},\\
\mc S_2&:=\{\sinh y=\sinh t \sin x\},\\
\hat S'_2&:=\{\cosh y=\cosh t \sin x,\,\, 0<x<\pi\}.
\end{align*}
The surfaces $\mc S_1$ and $\mc S'_1$ 
are given in  Goemans \cite[Figures 2.8 and 2,9]{G}
and also in Kim and Ogata \cite[Theorem 3.2]{O}
using different methods.
The surface $\mc S_2$
is given in
Akamine and Singh $($\cite[Figure 4 (b)]{A}$)$, 
and the surface $\mc S'_2$ is the same as in Kaya and Lopez 
$($\cite[Example 3]{Ka}$)$.
\end{remark}

\noindent
{\it Proof of Theorem~\ref{thm:173}.}\,
The surface given in \eqref{eq:692}
can be expressed as 
$F(x,y):=(f(x,y),x,y)$ ($x,y\in \R^2$),
where
$$
f(x,y):=\op{arcsinh}(e^{y}\cos x).
$$
It can be easily checked that $F$ has no umbilics.
Since any Kobayashi surface of order $n$ ($n\ge 2$)
has at least $2(n-2)$ umbilics (cf. \cite[(3.6)]{FKKRUY} at their
space-like parts, the only possibility is $n=2$.
Since
$$
1-f_x^2-y_y^2=\frac{2 - e^{2y} + e^{2y} \cos(2x)}{2 + e^{2y} + e^{2y} \cos(2x)},
$$
the set of space-like points of $F$ is given by
$$
D:=\left\{(x,y)\in \R^2\,;\, e^{2y}<\frac{2}{1-\cos{2x}}\right\},
$$
which is connected, and $\R^2\setminus D$ consists of infinite number
of time-like components, which makes
a contradiction, since 
every Kobayashi surface of order $2$ has 
at most four time-like components.
\qed

\begin{remark}\label{rmk:space-like}
Kobayashi \cite{Kob} gave two Weierstrass type formulas, one is
in \cite[(1.2)]{Kob} and the other is in \cite[(1.3)]{Kob} for space-like
maximal surfaces.
If we use the Weierstrass data
\eqref{eq:303}
as a pair of holomorphic functions,
then the first formula \cite[(1.2)]{Kob} produces a space-like 
ZMC-surface $G_1(u,v)$ whose image is a subset of  
\begin{equation}
\mc K_1:=\{(t,x,y)\in \R^3\,;\, e^t \cosh x=\cosh y\},
\end{equation}
which is an entire graph of mixed-type found by Kobayashi \cite{Kob},
and is a typical example of Kobayashi surfaces 
$($which is the case $n=2$ of
\cite[Example 3.6]{FKKRUY}$)$. 
The conjugate surface $G_2(u,v)$ 
of $G_1(u,v)$ is a subset of
\begin{equation}
\mc K_2:=\{(t,x,y)\in \R^3\,;\, \sin t=\sin x\sin y\},
\end{equation}
which is a triply periodic ZMC-surface with cone-like singular points
found by \cite{FKKRSTUYY}.
On the other hand,
the ZMC-surface $G_3(u,v)$
produced by the second formula \cite[(1.3)]{Kob}
is a subset of
\begin{equation}
\mc K_3:=\{(t,x,y)\in \R^3\,;\, \cos t \cosh x=\cos y\},
\end{equation}
which does not contain any time-like points.
The conjugate surface $G_4(u,v)$ 
of $G_3(u,v)$ is a subset of
\begin{equation}
\mc K_4:=\{(t,x,y)\in \R^3\,;\, \sinh t=-e^y \sin x\},
\end{equation}
which is an entire graph of mixed-type and is congruent to
$\mc S_4$ as above. So, one can prove 
Theorem~\ref{thm:173} using $\mc K_4$.
\end{remark}

\begin{remark}
Consider a graph 
$F(x,y)=(f(x,y),x,y)$, where $f$ is a real analytic function on $\R^2$.
If the mean curvature of $F$ identically vanishes on its time-like part $U$
$($resp. space-like part $V)$,
then $f$ satisfies $($cf. \cite{M} and \cite{Kob}$)$
\begin{equation}\label{eq:514}
(1-f_y^2)f_{xx}+2 f_xf_yf_{xy}+(1-f_x^2)f_{yy}=0
\end{equation} 
not only on on $U$ $($resp. $V)$
but also
on $\R^2$,  
since $f$ is real analytic.
In particular,
$F$ is a ZMC-graph on $\R^2$.
Moreover, we suppose that
$\Psi:\R^3_1\to \R$ is a real analytic function
and each connected component
of 
$$
Z(\Psi):=\{(t,x,y)\in \R^3_1\,;\, \Psi(t,x,y)=0\}
$$ 
contains a time-like $($resp. space-like$)$ ZMC-surface. 
Then the space-like $($resp. time-like$)$
parts of $Z(\Psi)$ also give ZMC-surfaces:
In fact, around a light-like regular point,
$Z(\Psi)$ can be locally expressed as a graph 
over a space-like plane in $\R^3_1$, and
the statement reduces to the case of 
the ZMC-graphs as above.
In particular,
$\mc E_4$, $\mc C_2$
$\mc S'_1$ and $\mc K_1$ 
are ZMC-surfaces 
on both space-like and time-like parts.
\end{remark}

\begin{acknowledgements}
  The authors would like to express 
  their deep gratitude to the reviewer, Shintaro Akamine, Shoichi Fujimori,
Atsufumi Honda, Naoya Ando, Kotaro Yamada and  Shunsuke Ichiki for valuable comments.
\end{acknowledgements}

\appendix
\section{Properties of para-holomorphic functions}

In this appendix, 
we introduce various properties of 
para-holomorphic functions.
Although the authors are aware that an excellent introduction of this 
subject has already been given in \cite{TKL},
the content of this paper is slightly different
(for example, we do not use \lq\lq para-complex limit", but 
apply only the para-Cauchy-Riemann equation as in \eqref{eq:PCR}).
The authors hope that 
this appendix with the reference \cite{TKL} 
will be a help of the readers' understanding.

Let $U$ be an open subset of $\check \C$ and $\phi:U\to \check \C$ 
a smooth function.
We can write 
$$
\phi(u+jv)=X(u,v)+j Y(u,v),
$$ 
where $X:=\re(\phi)$ and $Y:=\im(\phi)$. 
If we set
$$
\phi_z:=\frac12 (\phi_u+j \phi_v),\qquad
\phi_{\bar z}:=\frac12 (\phi_u-j \phi_v),
$$
then $\phi_{\bar z}$ vanishes  if and only if $\phi$ is
para-holomorphic (cf. \eqref{eq:PCR}). 
In this case, the differential of $\phi$
as a para-holomorphic function 
is defined by 
\begin{equation}\label{eq:P}
\phi':=X_u+j Y_u (=\phi_z).
\end{equation}
If $f_i$ ($i=1,2$) are smooth functions defined 
on open intervals $I_i(\subset \R)$ respectively.
Then
$$
\phi(u,v):=
\frac{f_1(u+v)+f_2(u-v)}2+j\frac{f_1(u+v)-f_2(u-v)}2
\qquad (u\in I_1,\,\, v\in I_2)
$$
is a typical example of para-holomorphic function.
Conversely, the following assertion holds:

\begin{lemma}\label{fact:A1}
Let $\phi$ be a para-holomorphic function on $U$.
Then, for each $z_0\in U$, there exist $\delta>0$
and two smooth functions $f_i:(-\delta,\delta)\to \R$
$(i=1,2)$ satisfying $f_i(0)=0$ such that 
$$
\{z_0+(u+v)+j(u-v)\in \check \C\,;\, |u|,|v|<\delta\}
$$
is contained in $U$ and
\begin{equation}\label{eq:1105}
\phi(u,v)=\phi(z_0)+\frac{f_1(u+v)+f_2(u-v)}2+j
\frac{f_1(u+v)-f_2(u-v)}2
\end{equation}
holds for each $u,v\in (-\delta,\delta)$.
\end{lemma}

\begin{proof}
We set $\phi=X+jY$, and
assume that $z_0=(0,0)$ and $\phi(z_0)=0$ without loss of generality.
Since $X_u=Y_v$ and $Y_v=X_u$, we have
$$
X_{uu}-X_{vv}=0,\qquad Y_{uu}-Y_{vv}=0.
$$
We set $u=x+y$ and $v=x-y$, then we have
$X_{xy}=Y_{xy}=0$.
So we can write
$X=\alpha_1(x)+\alpha_2(y)$ and $Y=\beta_1(x)+\beta_2(y)$,
where $\alpha_i$ and $\beta_i$ are smooth functions 
satisfying $\alpha_i(0)=\beta_i(0)=0$ for $i=1,2$. 
Applying \eqref{eq:PCR} and 
$x=(u+v)/2,\,\, y=(u-v)/2$, we have
\begin{align*}
&\frac{\alpha'_1(\frac{u+v}2)+\alpha'_2(\frac{u-v}2)}2=X_u=Y_v
=\frac{\beta'_1(\frac{u+v}2)-\beta'_2(\frac{u-v}2)}2 
\end{align*}
By setting $u=v$ or $v=-u$, we have
$\alpha'_1(u)=\beta'_1(u)$
and $\alpha'_2(u)=\beta'_2(u)$.
Since $\alpha_i(0)=\beta_i(0)=0$ for $i=1,2$,
we have $\alpha_1=\beta_1$ and $\alpha_2=-\beta_2$.
Setting $f_i(t):=\alpha_i(t/2)$ ($i=1,2$), we obtain
the conclusion.
\end{proof}

\begin{proposition}\label{label:prop1092}
Let $\phi$ be a para-holomorphic function on $U$
and set $x:=(u+v)/2,\,\, y:=(u-v)/2$.
Then there exists a pair of smooth functions 
$f_1(x)$ and $f_2(y)$ of one variable
such that
$$
\phi(u,v)=
\frac{f_1(u+v)+f_2(u-v)}2+j\frac{f_1(u+v)-f_2(u-v)}2
\qquad ((u,v)\in U).
$$
\end{proposition}

\begin{proof}
We set 
\begin{equation}\label{eq:1080}
\epsilon_1:=\frac{1+j}2,\qquad \epsilon_2:=\frac{1-j}2.
\end{equation}
Then it holds that
\begin{equation}\label{eq1E}
(\epsilon_1)^2=\epsilon_1,\quad (\epsilon_2)^2=\epsilon_2,\quad
\epsilon_1\epsilon_2=0
\end{equation}
and
\begin{equation}\label{eq2E}
\epsilon_1(s+jt)=\epsilon_1(s+t),\quad \epsilon_2(s+jt)=\epsilon_2(s-t)
\qquad (s,t\in \R),
\end{equation}
that is, the multiplications of $\epsilon_i$ ($i=1,2$) play the role 
of linear projections of $\check \C$ (thinking it as the Lorentz-Minkowski
$2$-plane) into the two light-like lines.
By the local expression
\eqref{eq:1105}, we have that
\begin{align}\nonumber
\phi(z)&=(\epsilon_1+\epsilon_2)\phi(z)
 \\ \label{eq:1110}
&=\epsilon_1 (c_1+f_1(u+v))+\epsilon_2 (c_2+f_2(u-v)),
\end{align}
where $c_i$ ($i=1,2$) are real numbers defined by
$$
c_1:=\op{Re}(\phi(p))+\op{Im}(\phi(p)),\qquad
c_2:=\op{Re}(\phi(p))-\op{Im}(\phi(p)).
$$
Since $\{\epsilon_1,\epsilon_2\}$ is a basis of $\check \C$ as a 
$2$-dimensional real vector space, $f_1(u+v)$ and $f_2(u-v)$ are uniquely determined. 
This implies that
the expression
\eqref{eq:1105} is global.
So replacing $(c_1+f_1,c_2+f_2)$ by $(f_1,f_2)$,
we obtain the conclusion.
\end{proof}

By definition, one can easily check the following (cf. \cite{Kon}):
\begin{enumerate}
\item Suppose that 
$\phi_k(z)$ ($k=1,2$)
are para-holomorphic functions on $U$.
Then $\phi_1+\phi_2$ and $\phi_1\phi_2$ are
para-holomorphic functions on $U$ and
$$
(\phi_1+\phi_2)'=\phi'_1+\phi'_2,\qquad 
(\phi_1\phi_2)'=\phi'_1\phi_2+\phi_1\phi'_2.
$$
\item Suppose that 
$\phi(z)$ is para-holomorphic on $U$.
If $\lnorm{\phi(z)}\ne 0$ for each $z\in U$, then $1/\phi$ is 
a  para-holomorphic function on $U$, and
satisfies  $(1/\phi)'=-\phi'/\phi^2$.
\item Let $U,V$ be open subsets on $\check \C$, and
$\phi:U\to V$ and $\psi:V\to \check C$ are
para-holomorphic functions.
Then $\psi\circ \phi$ is also para-holomorphic and satisfies
$(\psi\circ \phi)'(z)=\psi'(\phi(z))\phi'(z)$ on $U$.
\end{enumerate}
Since these statements can be shown in almost the same way, 
we show only (3). We set $\phi=X+jY$ and $\psi=A+jB$
where $X,Y$ and $A,B$ are real-valued functions.
If we set $\alpha:=\re{(\psi\circ \phi)}$ and $\beta:=\im{(\psi\circ \phi)}$.
Then 
$$
\alpha_u=A(X,Y)_u=A_XX_u+A_YY_u=B_YY_v+B_XX_v=B(X,Y)_v=\beta_v
$$
holds. Similarly, $\alpha_v=\beta_u$ is proved. So $\psi\circ \phi$ is
para-holomorphic on $U$, and
\begin{align*}
\alpha_u+j \beta_u&=
A_XX_u+A_YY_u+j(B_XX_u+B_YY_u) \\
&=A_XX_u+B_XY_u+j(B_XX_u+A_XY_u)
=(A_X+jB_X)(X_u+jY_u),
\end{align*}
proving the assertion.

We also show the following:

\begin{proposition}\label{eq:909}
Let $\phi(z)$ be a para-holomorphic function on a domain  $U$ in $\check \C$.
Then $\phi(z)dz$ is a closed $1$-form $U$. In particular, if $U$ is simply connected
and $z_0\in U$, then the line-integral
\begin{equation}\label{eq:914}
(\Phi(z)=)\int_{z_0}^z \phi(z)dz:=\int_\gamma (Xdu+Ydv)+j(Ydu+Xdv)\quad (z\in U)
\end{equation}
is well-defined, where $\gamma$ is a piece-wise smooth curve on $U$ 
from $z_0$ to $z$.
Moreover, $\Phi(z)$ is 
a para-holomorphic function satisfying $\Phi'(z)=\phi(z)$.
\end{proposition}

\begin{proof}
If we write $\phi(z)=X(u,v)+j Y(u,v)$, then we have
$$
\phi(z)dz=(Xdu+Ydv)+j(Ydu+Xdv).
$$
By
\eqref{eq:PCR},
the two
$1$-forms $Xdu+Ydv$ and $Ydu+Xdv$ are closed.
Moreover, since $U$ is simply connected,
the line-integral \eqref{eq:914}
is well-defined.
Then, the exterior derivative $d\Phi$ of
$\Phi$ as an $\R^2$-valued function
is equal to $(Xdu+Ydv,Ydu+Xdv)$.
So if we write $\Phi=A(u,v)+jB(u,v)$
where $A,B$ are smooth functions on $U$, then we have
$dA=Xdu+Ydv$ and $dB=Ydu+Xdv$, which imply that
$A_u=X$, $A_v=Y$ and $B_u=Y$, $B_v=X$.
So, we can conclude that $\Phi$
gives a para-holomorphic function on $U$ satisfying
$$
\Phi'=A_u+jB_u=X+jY=\phi,
$$
which proves the assertion.
\end{proof}

Since $\phi(z):=z$ ($z\in \check \C$)
is a para-holomorphic function,
any rational function of variable $z$ is a
para-holomorphic function on an open dense subset of $\check \C$. 
Let $I$ be an open subset of $\R$ (the set $I$ may or may not be an
interval on $\R$).
Consider a smooth function $f:I\rightarrow\R$.
We consider an open subset 
$$
U_I:=\{u+jv\in \check \C\,;\, u+v,u-v\in I\}.
$$
For $z=u+jv\in U_I$, we set
\begin{align}\label{eq;447}
  \tilde{f}(z):=\frac{f(u+v)+f(u-v)}2+j\frac{f(u+v)-f(u-v)}{2},
\end{align}
which is a para-holomorphic function on $U_I$.

\begin{remark}
To extend a real-valued function $h(u)$ on the real axis
to a holomorphic function, $h(u)$ must be a 
real analytic function, 
and its extension is given using a power series. 
However, i n order to extend $f(u)$  as a para-holomorhic function, 
the real analyticity of $f$ is not necessary, and
the extension $\tilde f$ can be defined by \eqref{eq;447}.
In particular, if $f(u)$ is defined on an open dense set $I$ in $\R$, 
then the domain $U_I$  of the definition $\tilde f(z)$ 
is also an open dense subset of $\check \C$.
\end{remark}

For example, by \eqref{eq;447}, we have
\begin{equation}
(\exp(z)=)e^z=\frac{e^{u+v}+e^{u-v}}2+j\frac{e^{u+v}-e^{u-v}}2
=e^u(\cosh v+j \sinh v).
\end{equation}
Using this, one can easily prove the following:
\begin{equation}
\exp(z)\exp(w)=\exp(z+w)\qquad (z,w\in \check \C).
\end{equation}
On the other hand, since
$\log|x|$ is defined on $I:=\R\setminus \{0\}$,
\begin{align}\label{eq:Log0}
\log z&:=
\log \sqrt{|\lnorm{z}|}+j\log\left(\sqrt{\frac{|u+v|}{|u-v|}}\right)
\end{align}
is a para-holomorphic function on $\check \C_*
:=\{z\in \check \C\,;\, \lnorm{z}\ne 0\}$.
One can directly check that
\eqref{eq:Log0} implies the formula
given in \eqref{eq:302} 
by dividing the case in terms of the signs of $u^2-v^2$,  $u$ and $v$.
Moreover,
\begin{equation}\label{eq:log}
\Big(\log(z\pm 1)\Big)'=\frac1{z\pm 1},\qquad
\Big(\log(z^2\pm 1)\Big)'=\frac{2z}{z^2\pm 1}
\end{equation}
hold whenever $z\pm 1\in \check \C_*$ or $z^2\pm 1\in \check \C_*$,
respectively.
One can easily check that 
\begin{equation}\label{eq:1247}
\log \circ \exp (z)=z \qquad (z\in \check \C)
\end{equation}
and
\begin{equation}\label{eq:1251}
\exp \circ \log (z)=z \qquad (\text{if $N^2(z)>0$ and $\re(z)>0$}).
\end{equation}
However,
$\exp \circ \log (z)=z$ does not hold in general.
For example,  
$$
\exp \circ \log(e^s(\sinh t+j\cosh t))=e^s(\cosh t+j \sinh t)
\qquad (s,t\in \R).
$$

One can easily check the following:
\begin{enumerate}
\item[(i)] Let $f_k$ ($k=1,2$) be two smooth functions 
defined on an open set $I$ in $\R$. Then 
$\widetilde{f_1+f_2}=\tilde f_1+\tilde f_2$ and 
$\widetilde{f_1f_2}=\tilde f_1\tilde f_2$ hold on $U_I$.
\item[(ii)] Let $f$ be a smooth function defined on an
open set $I$ in $\R$. If $f(u)\ne 0$ for each  $u\in I$,
then
$\widetilde{1/f}(z)=1/\tilde f(z)$ holds for $z\in U_I$.
\item[(iii)]
Let $f$ be a smooth function defined on an
open set $I$ in $\R$. 
If we set $g:=df/du$, then $\tilde{f}$ is para-holomorphic and
$(\tilde{f})'=\tilde g$ holds on $U_I$.
\end{enumerate}
We here prove only (iii):
Letting $\tilde{f}(z):=X(u,v)+jY(u,v)$, we have 
$X_u=Y_v$ and $X_v=Y_u$.
 Thus, $\tilde{f}(z)$ is para-holomorphic.
By \eqref{eq:P} with \eqref{eq;447}, we obtain 
$(\tilde{f})'=\tilde g$.

\begin{proposition}\label{prop:894}
Let $P(u)$ and $Q(u)$ be polynomials in $u$.
Then $\tilde f(z)$ 
is obtained by substituting $z$ into 
$f(u):=P(u)/Q(u)$, that is, 
$\tilde f(z)=f(z)$
holds whenever $\lnorm{Q(z)}\ne 0$.
\end{proposition}

\begin{proof}
If we set $f_0(u):=u$, then we have 
$$
\tilde f_0(z)=u+jv=f_0(u+jv).
$$
By (i), we have $\tilde P(z)=P(u+jv)$ and
$\tilde Q(z)=Q(u+jv)$ for each $z\in \check \C$.
If we set $I:=\{u\in \R\,;\, Q(u)\ne 0\}$,
then we have $U_I:=\{z\in \check \C\,;\, \lnorm{Q(z)}\ne 0\}$. 
By (ii), 
$\widetilde{1/Q}(z)=1/\tilde Q(z)$
holds for $z\in U_I$.
Applying (i) for $f(u):=P(u)$ and $g(u):=1/Q(v)$,
we obtain the equality.
\end{proof}

\begin{corollary}\label{1320}
Suppose that there exists 
$c\in I$
such that
$$
f(u)=\sum_{n=0}^\infty a_n (u-c)^n
$$
converges uniformly on $I$, 
where $\{a_n\}_{n=1}^\infty$ 
is a sequence of real numbers.
Then $\tilde f$ is 
a para-holomorphic function on $U_I$ satisfying
$$
\tilde f(z)=\sum_{n=0}^\infty a_n (z-c)^{n},\quad
\tilde f'(z)=\sum_{n=1}^\infty na_n (z-c)^{n-1}\qquad z\in U_I.
$$
\end{corollary}

\begin{proof}
We set 
$$
\phi_n(u):=\sum_{k=0}^n a_k (u-c)^k,\qquad \psi_n(u):=\sum_{k=1}^n k a_k (u-c)^{k-1} 
$$
for $u\in I$. Since convergence is uniform on $I$, we have
$$
f(u):=\lim_{n\to\infty}\phi_n(u),
\qquad f'(u):=\lim_{n\to\infty}\psi_n(u).
$$
So, it holds that
\begin{align*}
\tilde f(z)
&=\frac{f(u+v)+f(u-v)}2+j\frac{f(u+v)-f(u-v)}2 \\
&=\lim_{n\to \infty}\frac{\phi_n(u+v)+\phi_n(u-v)}2+j\frac{\phi_n(u+v)-\phi_n(u-v)}2 \\
&=\lim_{n\to \infty}\phi_n(z) = \sum_{n=0}^\infty a_n (z-c)^{n}.
\end{align*}
Applying this to the pair $(f',\psi)$, 
we have (cf. (iii))
$$
(\tilde{f})'(z)=\widetilde{f'}(z)=\sum_{n=1}^\infty n a_n (z-c)^{n}.
$$
\end{proof}

In \cite{Kon}, $\exp,\, \sinh$ and $\cosh$ 
as para-holomorphic functions are defined using power series. 
By the corollary,
they coincide with $\widetilde{\exp}$,
$\widetilde{\sinh}$ and $\widetilde{\cosh}$, respectively.
Moreover, we can apply Corollary \ref{1320} to the Maclaurin expansions of
$\sin u$ and $\cos u$ ($u\in \R$), 
and have the identities
\begin{equation}
\widetilde{\sin}\, z=\sum_{n=0}^\infty \frac{(-1)^nz^{2n+1}}{(2n+1)!},
\quad
\widetilde{\cos}\, z=\sum_{n=0}^\infty \frac{(-1)^nz^{2n}}{(2n)!}
\qquad (z\in \check\C).
\end{equation}
In this paper, 
$\tilde \exp ,\, \tilde \sin,\tilde \cos, \tilde \sinh$ and $\tilde \cosh$ 
are simply
written as $\exp ,\, \sin,\cos, \sinh$ and $\cosh$ 
unless there is confusion. 
It should be remarked that 
\begin{equation}
\frac{\tilde e^{jz}+\tilde e^{-jz}}2=\widetilde{\cosh}\, z,\quad
\frac{\tilde e^{jz}-\tilde e^{-jz}}2=\widetilde{\sinh}\, z
\qquad (z\in \check \C)
\end{equation}
hold, and $\sin,\cos$ do not appear in these identities.

We consider the case that 
$f$ has a smooth inverse $g:=f^{-1}:f(I)\to I$.
We set
$$
V_I:=\{u+jv\in \check \C\,;\, u+v,u-v\in f(I)\}.
$$
Then the function
\begin{align}\label{eq;737}
  \tilde{g}(z):=\frac{f^{-1}(u+v)+f^{-1}(u-v)}2+j\frac{f^{-1}(u+v)-f^{-1}(u-v)}{2}
\end{align}
is well-defined for $(z:=)u+jv\in V_I$.
Moreover, the following assertion holds:

\begin{proposition}\label{prop;inverse}
Let 
$f:I\to f(I)(\subset \R)$ be
a diffeomorphism 
defined on
an open subset $I$ of $\R$, 
and let $g:f(I)\to I$
be the inverse map.
Then, $\tilde g:V_{I}\to \check \C$
is a para-holomorphic function giving
the inverse map of $\tilde f:U_I\to \C$.
Moreover,
\begin{align}\label{eq;550}
  (\tilde{f}^{-1})^\prime(w)=\tilde g'(w)=\tilde h(w)
\qquad (w\in V_I)
  \end{align}
holds, where $h:=dg/du$.
\end{proposition}

\begin{proof}
It can be easily checked that $\tilde f:U_I\to \check \C$ is an injection
and  $f(U_I)\subset V_I$.
By using
\eqref{eq1E},
it can be also checked that
$\tilde g\circ \tilde f(u+jv)=u+jv$ when $u+jv \in U_I$, which implies that
$\tilde g$ is the inverse map of $\tilde f$.
So, it is sufficient to prove \eqref{eq;550}. 
Since $h=dg/du$,
differentiating \eqref{eq;737} by $u$ according to \eqref{eq:P},
we have \eqref{eq;550}.
\end{proof}

For example, 
by Proposition~\ref{prop:894}, 
$h(u):=(1+u^2)^{-1}(=(\arctan u)^\prime)$
satisfies $\tilde h(z)=(1+z^2)^{-1}$.
By Proposition~\ref{prop;inverse}, we obtain
\begin{equation}\label{eq:TanI}
  (\arctan z)^\prime=\cos^2(\arctan z)=\frac{1}{1+z^2}
\qquad (z\in \check \C).
\end{equation}
Since $z=u+jv$ ($u,v\in \R$) satisfies
\begin{equation}\label{eq:N3}
\lnorm{1+z^2}=\Big(1+(u-v)^2\Big)\Big(1+(u+v)^2\Big)>0 \qquad (z\in \check \C),
\end{equation}
the right-hand side of \eqref{eq:TanI} is well-defined.

\end{document}